\documentclass[finale]{siamart1116}

\usepackage{cite}
\usepackage{todonotes}
\usepackage{multirow}

\usepackage{lipsum}
\usepackage{amsfonts}
\usepackage{graphicx}
\usepackage{epstopdf}
\ifpdf
  \DeclareGraphicsExtensions{.eps,.pdf,.png,.jpg}
\else
  \DeclareGraphicsExtensions{.eps}
\fi

\numberwithin{theorem}{section}

\newcommand{\TheTitle}{Proximally Guided Stochastic Subgradient Method for Nonsmooth, Nonconvex Problems} 
\newcommand{\TheAuthors}{D. Davis and B. Grimmer}

\headers{Proximally Guided Stochastic Subgradient Method}{\TheAuthors}

\title{{\TheTitle}\thanks{Submitted to the editors 9/11/2017.
\funding{This material is based upon work supported by the National Science Foundation under Award No. 1502405 and by the National Science Foundation Graduate Research Fellowship under Grant No. DGE-1650441.}}}

\author{
  Damek Davis\thanks{Cornell University, Ithaca, NY
    (\email{dsd95@cornell.edu}, \url{https://people.orie.cornell.edu/dsd95}).}
  \and
  Benjamin Grimmer\thanks{Cornell University, Ithaca, NY (\email{bdg79@cornell.edu},
	\url{https://people.orie.cornell.edu/bdg79}).}
}

\usepackage{amsopn}

\ifpdf
\hypersetup{
  pdftitle={\TheTitle},
  pdfauthor={\TheAuthors}
}
\fi

\usepackage{etoolbox}
\usepackage{mathtools}

\newcommand\numberthis{\addtocounter{equation}{1}\tag{\theequation}}

\DeclarePairedDelimiter{\dotp}{\langle}{\rangle}

\usepackage{booktabs}

\usepackage{mathtools}
\usepackage[normalem]{ulem} % to use \sout
\newcommand{\remove}[1]{{}}

\newcommand{\cB}{{\mathcal{B}}}
\newcommand{\cC}{{\mathcal{C}}}

\newcommand{\cF}{{\mathcal{F}}}

\newcommand{\cX}{{\mathcal{X}}}

\newcommand{\RR}{\mathbb{R}}
\newcommand{\NN}{\mathbb{N}}

\newcommand{\EE}{\mathbb{E}}

\newcommand{\dom}{{\mathrm{dom}}} % domain
\newcommand{\minimize}{\text{minimize}}

\newcommand{\bc}{\begin{center}}
\newcommand{\ec}{\end{center}}

\newcommand{\bdm}{\begin{displaymath}}
\newcommand{\edm}{\end{displaymath}}

\newcommand{\beq}{\begin{equation}}
\newcommand{\eeq}{\end{equation}}

\newcommand{\bfl}{\begin{flushleft}}
\newcommand{\efl}{\end{flushleft}}

\newcommand{\bt}{\begin{tabbing}}
\newcommand{\et}{\end{tabbing}}

\newcommand{\beqn}{\begin{align}}
\newcommand{\eeqn}{\end{align}}

\newcommand{\beqs}{\begin{align*}} % no equation numbers
\newcommand{\eeqs}{\end{align*}}  % no equation numbers

\newtheorem{assumption}{Assumption}

\newtheorem{example}{Example}
\newtheorem{remark}{Remark}

\newcommand{\prox}[3] {\mathbf{prox}_{#1 #2}(#3)}

\newcommand{\argmin}{\mathrm{argmin}}
\newcommand{\PP}{\mathbb{P}}
\newcommand{\dist}{\mathrm{dist}}
\newcommand{\epi}{\mathrm{epi}}
\newcommand{\proj}{\mathrm{proj}}
\newcommand{\sign}{\mathrm{sign}}
\usepackage{algpseudocode}
\newcommand*\Let[2]{\State #1 $\gets$ #2}
\algrenewcommand\algorithmicrequire{\textbf{Input:}}
\algrenewcommand\algorithmicensure{\textbf{Output:}}

\begin{document}

\maketitle

% REQUIRED
\begin{abstract}
	%In this paper, we introduce a stochastic projected subgradient method for weakly convex (i.e., uniformly prox-regular) nonsmooth, nonconvex functions---a wide class of functions which includes the additive and convex composite classes. At a high-level, the method is an inexact proximal point iteration in which the strongly convex proximal subproblems are quickly solved with a specialized stochastic projected subgradient method. The primary contribution of this paper is a simple proof that the proposed algorithm converges at the same rate as the stochastic gradient method for smooth nonconvex problems. This result validates the use of stochastic subgradient methods in nonsmooth, nonconvex optimization as is common when optimizing neural networks. 
In this paper, we introduce a stochastic projected subgradient method for weakly convex (i.e., uniformly prox-regular) nonsmooth, nonconvex functions---a wide class of functions which includes the additive and convex composite classes. At a high-level, the method is an inexact proximal point iteration in which the strongly convex proximal subproblems are quickly solved with a specialized stochastic projected subgradient method. The primary contribution of this paper is a simple proof that the proposed algorithm converges at the same rate as the stochastic gradient method for smooth nonconvex problems. This result appears to be the first convergence rate analysis of a stochastic (or even deterministic) subgradient method for the class of weakly convex functions. In addition, a two-phase variant is proposed that significantly reduces the variance of the solutions returned by the algorithm. Finally, preliminary numerical experiments are also provided. 
\end{abstract}

% REQUIRED

% REQUIRED

\section{Introduction}\label{sec:Introduction}
%% What are we trying to do?
Stochastic approximation methods iteratively minimize the expectation of a family of  known loss functions with respect to an unknown probability distribution. Such methods are of fundamental importance in machine learning, signal processing, statistics, and data science more broadly. For example, in machine learning, one is often interested in designing a classifier that performs well on the entire population of samples, given only a finite list of correctly labeled pairs $z_1, \ldots, z_n$ obtained from a fixed, but otherwise unknown distribution $\PP$. Mathematically, such problems may be formulated as \emph{population risk minimization}: 
\begin{align}\label{eq:mainprob}
\minimize~ F(x) := \begin{cases} \EE_{z \sim P}\left[ f(x, z)\right] & \text{if $x \in \cX$;}\\
\infty & \text{otherwise.}
\end{cases}
\end{align}
where $(\Omega, \cF, \PP)$ is a probability space, the set $\cX\subseteq \RR^d$ is closed and convex, and $f : \RR^d \times \Omega \rightarrow \RR$ is a known loss function, which encodes the loss of decision rule $x \in \RR^d$ on sample $z \in \Omega$. 

Much algorithmic development has been inspired by~\eqref{eq:mainprob}.  Robbins-Monro's pioneering work 1951 work~\cite{MR0042668} developed the first method for solving~\eqref{eq:mainprob} when each $f(\cdot, z)$ is smooth and strongly convex and $\cX = \RR^d$. This and most later methods are variants of the stochastic projected (sub)gradient method, which iteratively constructs approximate solutions $x_t$ of~\eqref{eq:mainprob} through the recursion
\begin{align*}
&\text{Sample} ~ z_t \sim \PP\\
&\text{Set} ~x_{t+1} = \proj_{\cX}(x_t - \alpha_t \nabla_x f(x_t, z_t)),
\end{align*}
where $z_1, \ldots, z_t, \ldots $ are IID and $\alpha_t$ is an appropriate control sequence. For nonsmooth $f(\cdot, z_t)$, sample subgradients are simply replaced by sample subgradients $v_t \in \partial f(x_t, z_t)$, where $\partial f(x_t, z_t)$ denotes the subdifferential in the sense of convex analysis~\cite{rockafellar2015convex}. 

The complexity of minimizing~\eqref{eq:mainprob} is directly related to the regularity of $f(\cdot, z)$. For example, for convex functions $f(\cdot, z)$ the stochastic subgradient method attains expected functional accuracy $\varepsilon$ with after $O(\varepsilon^{-2})$ stochastic subgradient evaluations. For strongly convex losses, the number of stochastic subgradient evaluations drops to $O(\varepsilon^{-1}).$ The interested reader may turn to the seminal work~\cite{complexity} for an in-depth investigation of these methods and for information-theoretic lower bounds showing such rates are unimprovable without further assumptions. 

For convex functions, complexity theory does not favor smooth losses over nonsmooth losses. For nonconvex problems, the situation is less clear. In the smooth case, the seminal work of Ghadimi, Lan, and Zhang~\cite{Ghadimi2016mini} develops a variant of the stochastic projected gradient method and establishes that the expected norm of the projected gradient
\begin{align}\label{eq:projected_gradient}
\EE_{z_1, \ldots, z_t}\left[\| x_t - \proj_{\cX}(x_t -  \nabla_x\EE_{z \sim P}\left[ f(x_t, z)\right])\|^2\right],
\end{align}
a natural measure of stationarity, tends to zero at a controlled rate. Namely, with $O(\varepsilon^{-2})$ stochastic gradient evaluations, the algorithm produces a point with expected projected gradient norm squared less than $\varepsilon$.   

At the time of writing the original version of this manuscript, there was no similar rate of convergence in the nonsmooth nonconvex setting for any known subgradient-based algorithm. Part of the difficulty in establishing a complexity theory for nonsmooth nonconvex subgradient-based methods is that the ``usual criteria," namely the  objective error and the norm of the gradient, can be completely meaningless. Indeed, on the one hand, one cannot expect the objective error $F(x_t) - \inf F$ to tend to zero---even in the smooth setting. On the other hand, simple examples, e.g., $F(x) = |x|$, show that 
$
\dist(0, \partial F(x_t))
$
can be strictly bounded below by a fixed constant for all $t$. 

In contrast to subgradient-based methods, the ``usual criteria" is meaningful for the \emph{proximal point method}~\cite{PPA}, which constructs a sequence $x_t$ of approximate minimizers through the recursion 
\begin{align*}
x_{t+1} &= \argmin_{x \in \RR^d} \left\{ F(x) + \frac{1}{2\gamma}\|x - x_t\|^2\right\},
\end{align*}
where $\gamma$ is a control parameter. Namely, it is a simple exercise to show that under minimal assumptions on $F$, the subdifferential distance $\dist(0, \partial F(x_{t}))$ tends to zero. Of course, each step of the proximal point method is difficult, if not impossible to execute without further assumptions on $F$. 

The search for an appropriate class of functions $F$ for which each proximal subproblem may be (approximately) executed naturally leads us to the deceptively simple, yet surprisingly broad class of $\rho$-\emph{weakly convex} functions. This is the class of functions that become convex after adding the quadratic $\frac{\rho}{2}\|\cdot\|^2$. For example, any $C^2$ function on a compact, convex set becomes convex after adding the quadratic $\tfrac{|\lambda_+|}{2}\|\cdot\|^2$, where $\lambda$ is the minimal eigenvalue of its Hessian across all points in the set.  In the nonsmooth setting, this class includes all \emph{convex composite losses}
$$
h(c(x))
$$
where $h$ is convex and $L$-Lipschitz and $c$ is $C^1$ with $\beta$-Lipschitz Jacobian; such functions are known to be $\beta L$-weakly convex~\cite[Lemma 4.2]{drusvyatskiy2016accelerated}. The additive composite class is another widely used, much studied class of weakly convex functions, formed from all sums 
\begin{align*}
g(x) + r(x)
\end{align*}
where $r$ is closed and convex and $g$ is $C^1$ with $\beta$-Lipschitz gradient; such functions are known to be $\beta$-weakly convex. For further examples of weakly convex functions, see~\cite[Section 2.1]{davis2018stochastic}, which includes formulations of robust phase retrieval, covariance matrix estimation, blind deconvolution, sparse dictionary learning, robust principal component analysis, and conditional value at risk. We provide several further examples in Section~\ref{sec:examples_weakly}. It is important to note that none of these applications are covered by the seminal work of Ghadimi, Lan, and Zhang~\cite{Ghadimi2016mini}, which assumes an additive composite objective form. 

\textbf{Contributions.} In this paper, we develop the first known complexity guarantees for a subgradient-based method for a general class of nonsmooth nonconvex losses in stochastic optimization. The guarantees in this paper apply to $\rho$-weakly convex losses $F$. Our algorithm, called the Proximally Guided stochastic Subgradient Method (PGSG) (Algorithm~\ref{alg:PGSG}), follows an inner-outer loop strategy that may be compactly and informally summarized 
\begin{align}\label{eq:inner-outer}
x_{t+1} &= \text{$\varepsilon$-}\argmin_{x \in \RR^d} \left\{ F(x) + \rho\|x - x_t\|^2\right\} \qquad(\text{in expectation}).
\end{align}
 The outer loop of PGSG is governed by the approximate proximal point method applied to the population risk $F$. Due to $\rho$-weak convexity of $F$, the inner loop subproblem is a \emph{strongly convex} stochastic optimization problem. Thus, by classical complexity theory, approximate solutions to the inner loop subproblems may be quickly found. When both inner and outer loops are coupled together appropriately, we establish that this method produces a point $\bar x$ that is $\varepsilon$-close in expectation to the set of $\varepsilon$-critical points after $O(\varepsilon^{-2})$ stochastic subgradient evaluations, meaning, 
\begin{align}\label{eq:nearbynearlystationary}
\EE\left[\dist(\bar x, \{ x \mid \dist(0, \partial F(x))^2 \leq \varepsilon\} )^2\right] \leq \varepsilon,
\end{align}
where $\partial F$ is denotes the subdifferential of $F$ in the sense of variational analysis~\cite{rtrw}; see Section~\ref{sec:Inexact}. The nearly stationary point nearby $\bar x$ is it self a solution to a \emph{strongly convex} stochastic optimization problem so it is in principle obtainable to any desired degree of accuracy; see Remark~\ref{remark:finding_stat}. As stated before, simple examples show that one cannot expect the iterates produced by a subgradient-based algorithm themselves to be $\varepsilon$-stationary because $\dist(0, \partial F(x_t))$ may be bounded below for all $t$. Thus, in some sense this is the most natural convergence criteria available.

Having established expectation guarantees, we turn our attention to probabilistic guarantees. Namely, following~\cite{ghad} (which considers the smooth case), we say a (random) point $\bar x$ is an \emph{$(\varepsilon, \Lambda)$-solution} if 
$$
\PP\left(\dist(\bar x, \{ x \mid \dist(0, \partial F(x))^2 \leq \varepsilon\})^2 \leq \varepsilon  \right) \geq 1- \Lambda.
$$
Markov's inequality shows that PGSG finds an $(\varepsilon, \Lambda)$-solution $\bar x$ after 
\begin{align*}
O\left(\frac{1}{\Lambda^2\varepsilon^2}\right)
\end{align*}
stochastic subgradient evaluations. To improve this complexity, we introduce a 2-phase algorithm, called 2PGSG, which produces an $(\varepsilon, \Lambda)$-solution after 
\begin{align*}
O\left(\frac{\log(1/\Lambda)}{\varepsilon^2} +  \frac{\log(1/\Lambda)}{\Lambda\varepsilon}\right),
\end{align*}
stochastic subgradient evaluations, substantially reducing the variance of our solution estimate. The technique for achieving this improvement is somewhat different than what~\cite{ghad} proposes in the smooth case. The challenge in establishing the result is that we no longer have unbiased estimates of subgradients at nearly stationary points. Indeed, the iterates produced by subgradient methods are only nearby nearly stationary points and are not nearly stationary themselves.

Finally, we turn our attention to a more practical variant of PGSG, which does not assume that the weak convexity constant $\rho$ is known. In this setting, a simple idea---letting the outer loop stepsize tend to infinity---results in a point $\bar x$, which satisfies~\eqref{eq:nearbynearlystationary} after $O(\varepsilon^{2/(1-\beta)})$ stochastic subgradient evaluations, where $\beta \in (0, 1)$ is a user defined meta-parameter. We mention that the seminal work of  Ghadimi, Lan, and Zhang~\cite{Ghadimi2016mini} also assumes knowledge of the weak convexity constant $\rho$; in their setting $\rho$ is simply the Lipschitz constant of the gradient. 

We validate our results with some preliminary numerical experiments on the population objective of a robust real phase retrieval problem. We also discuss several more examples of Weakly convex functions in Section~\ref{sec:examples_weakly}.

\subsection{Related Work}\label{sec:related}
\paragraph{Stochastic Gradient Methods} The convergence rates presented in  match known rates for the stochastic gradient method in nonconvex optimization~\cite{ghad}. There, the standard stochastic gradient method may be used without modification. Interestingly, recent work has developed methods, which converge at the improved rate of $O(\varepsilon^{-3/2})$~\cite{fang2018spider}, showing a surprising gap between smooth and nonsmooth nonconvex optimization not present in the convex case. 

\paragraph{Stochastic Proximal-Gradient Methods}\label{sec:related_work} For additive composite problems
$$
\minimize \left\{\EE_z\left[f(x, z)\right] + r(x)\right\},
$$
one often employs stochastic proximal-gradient methods, which require, at every iteration, a (potentially costly) evaluation of the mapping $\prox{}{ r}{y} = \argmin \{ r(x) + \tfrac{1}{2}\|x - y\|^2\}$. These methods achieve expected projected gradient norm $\varepsilon$, as in~\eqref{eq:projected_gradient}, after $O(\varepsilon^{-2})$ stochastic gradient evaluations~\cite{Ghadimi2016mini}. These methods have also been extended to regularizers that are arbitrary closed prox-bounded functions $r$ \cite{wotao}, a setting which we do not recover. 

Evaluating the proximal mapping of $r$ could be substantially more expensive than computing a subgradient. For example, if $r = \|\cdot \|_2$ is the spectral norm on $\RR^{n \times n}$, then its proximal mapping requires a full singular value decomposition. In contrast, a subgradient may be computed from a single maximal eigenvector.%, which is a less expensive operation. 

Another advantage of stochastic subgradient methods over stochastic proximal-gradient methods, is that multiple nonsmooth functions may be present in the objective function $F$. The same is not true for stochastic proximal-gradient methods:  even if two functions $r_1$ and $r_2$ have simple proximal operators, the proximal operator of the sum $r= r_1 + r_2$ can be quite complex. Similarly, the proximal operator of an expectation $\EE_z\left[r(x, z)\right]$ could be intractable.

\paragraph{Stochastic Methods for Convex Composite} Recently~\cite{duchi2017stochastic} proposed a method for finding stationary points of the convex composite problem in which $f(x, z)= h(c(x, z), z)$. The first method adapts the prox-linear algorithm~\cite{Burke1987,Burke1995,drusvyatskiy2016nonsmooth,burke1985descent,drusvyatskiy2016error,Lewis2016,Fletcher1982} to the stochastic setting: given $x_t$, sample $z_t$ and form $x_{t+1}$ as the solution to the convex problem:
\begin{align}\label{eq:proxlinear}
x_{t+1} = \argmin_{x \in X} \left\{ h(c(x_t, z_t) + \nabla c(x_t, z_t)(x - x_t), z_t) +  \frac{1}{2\gamma_t} \|x - x_t\|^2\right\},
\end{align}
where $\gamma_t = \theta(1/\sqrt{t})$. The second proposed method is a straightforward application of the stochastic projected subgradient method~\cite{doi:10.1137/070704277}. Both methods are shown to almost surely converge to stationary points, but no rates of convergence are given. %In this paper, we do not provide a convergence rate for the prox-linear algorithm~\eqref{eq:proxlinear}, but we do provide a convergence rate for PGSG, which is a slight modification of the stochastic projected subgradient method analyzed in~\cite{duchi2017stochastic}. When each prox-linear subproblem is costly, stochastic subgradient methods are preferable to prox-linear methods because subgradients of $h( c(x, z),z) $ are easy to compute (see Section~\ref{sec:instantiations}).%$\partial_F (h^\xi \circ c^\xi)(x) := \nabla c^\xi(x)^T \partial h^\xi(c^\xi(x))$.

We remark that the convergence proof presented in~\cite{duchi2017stochastic} is complex, being based on the highly nontrivial theory of nonconvex differential inclusions. We believe there is a benefit to having a simple proof of convergence, albeit for a slightly different subgradient method, which is what we provide in this paper.

Further work on minimizing convex composite problems appears in~\cite{wang2017variance-reduced,wang2016prox-gradient,wang2017general}. This series of papers analyzes nested expectations: $F(x)=\EE_v[h(\EE_w[c(x, w)\mid v], v)]$. Although the stochastic structure considered in these papers is more general than what we consider in Problem~\ref{eq:mainprob}, the assumptions made on $F$ are much stronger than our assumptions on $F$. In particular, the authors prove  rates under the assumption that (a) $F$ is convex, (b) $F$ is strongly convex, or (c) $F$ is nonconvex, but \emph{differentiable} with Lipschitz continuous gradient. For case (c), the authors propose an algorithm that finds an $\varepsilon$-stationary point of $F$ after $O(\varepsilon^{-2.25})$ gradient evaluations \cite{wang2016prox-gradient} (in particular, they consider unconstrained problems). In contrast, we find a point that is $\varepsilon$-close to an $\varepsilon$-stationary point of the nonsmooth, nonconvex function $F$ after $O(\varepsilon^{-2})$ subgradient evaluations.

%Under the additional assumptions that $F$ has Lipschitz continuous gradient and $X=\RR^d$, they show that their accelerated stochastic gradient method finds an $\varepsilon$-stationary point after $O(\varepsilon^{-9/4})$ gradient evaluations \cite{wang2016prox-gradient}. Their convergence rate improves in the special case of $F$ being strongly convex. The authors also consider a variance reduced approach to this problem in the case of finite-sums, under the strong assumptions that $F$ has Lipschitz continuous gradient and is strongly convex, in~\cite{wang2017variance-reduced}.

\paragraph{Inexact Proximal Point Methods in Nonconvex Optimization} The idea of using the inexact proximal point method to guide a nonconvex optimization algorithm to stationary points is not new. For example, Hare and Sagastizabal~\cite{Hare2009,Hare2010} propose a method for computing inexact proximal points, which then enables the analysis of a nonconvex bundle method. The more recent work~\cite{paquette2017catalyst} exploits linearly convergent algorithms for solving the proximal subproblems. In contrast for the subproblems considered in this work, there are no linearly convergent stochastic subgradient algorithms capable of minimizing the proximal point step.

\paragraph{Subgradient Methods for Weakly Convex Problems} This paper is not the first to consider subgradient methods under weak convexity. For example, the early work~\cite{Nurminski1979} proves subsequential convergence of the (non projected) subgradient method for weakly convex \emph{deterministic} problems. However, no rates were given in that work. 

\paragraph{Almost Sure Convergence of Stochastic Subgradient Methods for Nonconvex Problems} Convergence to stationary points of stochastic subgradient methods in nonsmooth, nonconvex optimization has previously been attained under several different scenarios, some of which are more general than the scenario considered in Problem~\eqref{eq:mainprob}~\cite{10.2307/3689671,Ermol'ev1998,Ermoliev2003}. No rates of convergence were given in these works. In contrast, the novelty of the proposed approach lies in the attained rate of convergence, which matches the best known rates of convergence for smooth, nonconvex stochastic optimization~\cite{ghad}.

\paragraph{Rates of Convergence in Stochastic Weakly Convex Optimization} Since the first draft of this paper appeared on arXiv in July 2017, several works appearing in 2018 have established convergence of the standard stochastic projected subgradient method under weak convexity~\cite{davis2018stochastic,davis2018stochastic_sub}. The obtained rates (in expectation) are essentially the same as those obtained in this paper, namely they are of the form presented in equation~\eqref{eq:nearbynearlystationary}. The authors of~\cite{davis2018stochastic,davis2018stochastic_sub} do not provide any probabilistic guarantees.

\subsection{Outline}

Section~\ref{sec:not} presents notation and several basic results used in this paper, as well as further examples of weakly convex functions. Section~\ref{sec:Inexact} presents our convergence analysis under the assumption that $\rho$ is known. Section~\ref{sec:proabilistic} presents our probabilistic guarantees. Section~\ref{sec:weakconvexityunknown} presents our convergence analysis when $\rho$ is unknown.  Section~\ref{sec:experiments} preliminary presents numerical results obtained on a robust phase retrieval problem.

\section{Notation and Basic Results} \label{sec:not}
Most of the notation and concepts we use in this paper can be found in~\cite{rtrw,doi:10.1137/1.9781611971309}. Our main probabilistic assumption is that we work in a probability space $(\Omega, \cF, \PP)$ and $\RR^d$ is equipped with the Borel $\sigma$-algebra, which we use to define measurable mappings. 

For a given function $F: \RR^d \rightarrow \RR \cup \{\infty\}$, we let 
$$
\dom ~F = \{ x\in \RR^d  \mid F(x) < \infty \} \qquad \epi ~F = \{(x, t)\in \RR^d \times \RR \mid f(x) \leq t\}.
$$
We say a function is \emph{closed} if $\epi ~F$ is a closed set. We say a function is \emph{proper} if $\dom~F \neq \emptyset$.

Let $F \colon \RR^d \rightarrow \RR\cup \{\infty\} $ be a proper closed function. At any point $x \in \dom~ F$, we let 
\begin{align*}
\partial F(x) =  \{ v \in \RR^d \mid \left(\forall y \in \RR^d\right) \; F(y) \geq F(x) + \dotp{v, y-x} + o(\|y - x\|)\}
\end{align*}
denote the \textit{Fr{\'e}chet subdifferential} of $F$ at $x$. On the other hand, if $x \notin \dom~ F$ we let $\partial F(x) = \emptyset$. It is an easy exercise to show that at any local minimizer $x$ of $F$, we have the inclusion $0 \in \partial F(x)$. 

For the class of weakly convex functions, all elements of the subdifferential generate quadratic underestimators of the function $F$, as the following proposition shows. The equivalences are based on \cite[Theorem 3.1]{daniilidis2005filling}.
%The following calculation, which is based on \cite[Theorem 3.1]{daniilidis2005filling} and standard results in variational analysis, enables us to compute the subdifferentials of $\obj^\xi$.
\begin{proposition}[Subgradients of Weakly Convex Functions]\label{prop:generalweak}
Suppose that $F : \RR^d \rightarrow \RR \cup \{ \infty\}$ is a closed function. Then the following are equivalent
\begin{enumerate}
\item $F$ is $\rho$-weakly convex. That is, $F + \frac{\rho}{2} \|\cdot\|^2$ is convex. 
\item\label{prop:generalweak:item:perturbed} For any $x, y \in \RR^d$ with $v \in \partial F(x)$, we have
\begin{align}\label{eq:perturbedconvexity}
F(y) \geq F(x) + \dotp{v, y - x} - \frac{\rho}{2}\|y - x\|^2.
\end{align}
\item For all $x, y \in \RR^d$ and $\alpha \in [0, 1]$, we have 
\begin{align*}
F(\alpha x + (1-\alpha)y) \leq \alpha F( x) + (1-\alpha)f(y) + \frac{\rho\alpha(1-\alpha)}{2}\|x-y\|^2.
\end{align*}
\end{enumerate}
%Moreover, assuming $F$ is $\rho$-weakly convex, for any $x \in \dom~F$ and $y \in \RR^d$, we have the identity 
%\begin{align*}
%\partial \left(F(\cdot) + \frac{\rho}{2}\|\cdot -y \|^2\right)(x) = \partial F(x) + \rho (x-y).
%\end{align*}
\end{proposition}

\subsection{Examples of Weakly Convex Functions} \label{sec:examples_weakly}

As stated in the introduction, the class of weakly convex functions is broad.  In the nonsmooth setting, this class includes all \emph{convex composite losses}
$$
h(c(x))
$$
where $h$ is convex and $L$-Lipschitz and $c$ is $C^1$ with $\beta$-Lipschitz Jacobian; such functions are known to be $\beta L$-weakly convex~\cite[Lemma 4.2]{drusvyatskiy2016accelerated}. Several popular weakly convex formulations are presented in~\cite[Section 2.1]{davis2018stochastic}. We now discuss several further examples. 

\begin{example}[Censored Block Model]
The censored block model~\cite{abbe2014censored} is a variant of the standard stochastic block model~\cite{abbe2016exact}, which seeks to detect two communities in a partially observed graph. Mathematically, we encode such communities by forming the ``community matrix" $M = \bar \theta \bar \theta^T \in \{-1, 1\}^d$, where $\bar x \in \{-1, 1\}^d$ is a membership vector in which $\bar x_i = 1$ if node $i$ is in the first community, and $\bar x_i = -1 $ otherwise. In the censored block model, we observe a randomly corrupted version $\hat M$ of the matrix $M$
$$
\hat M = \begin{cases} 
0 & \text{ with probability $1-p$;}\\
M_{ij} & \text{with probability $p(1-\epsilon)$;}\\
-M_{ij} & \text{with probability $p\epsilon$.}\\
\end{cases}
$$
Then our task is to recover $M$ given only $\hat M$. We may formulate this problem in the following form convex, composite form:
\begin{align*}
F(x) = \sum_{ij \mid \hat M_{ij} \neq 0} | x_i x_j - \hat M_{ij}|.
\end{align*}
Notice that absolute value function encourages the matrix to $xx^T$ agree with $\hat M$ in most of its nonzero entries---the bulk of which are equal to $M_{ij}$---due to the sparsity promoting behavior of the nonsmooth absolute value function. 
\end{example}

\begin{example}[Robust Phase Retrieval]
Phase retrieval is a common task in computational science; applications include imaging, X-ray crystallography, and speech processing. Given a set of tuples $\{(a_i,b_i)\}_{i=1}^m\subset\RR^d\times \RR$, the (real) phase retrieval problem seeks a vector $x\in \RR^d$ satisfying $(a_i^Tx)^2=b_i$ for each index $i=1,\ldots,m$. This problem is NP-hard \cite{phae_NPhard}. Strictly speaking, phase retrieval is a feasibility problem. However, when the set of measurements $\{b_i\}$ is corrupted by gross outliers, one considers the following ``robust" phase retrieval objective: 
\begin{align*}
F(x) = \frac{1}{n}\sum_{i=1}^n |\dotp{a_i, x}^2 - b_i|.
\end{align*}
Notice that this nonsmooth objective is given in convex composite form, and therefore, it is weakly convex. 
\end{example}

\begin{example}[Nonsmooth Trimmed Estimation]
Let $f_1, \ldots, f_n$ be Lipschitz continuous, convex loss functions on $\RR^d$. The goal of trimmed estimation~\cite{rousseeuw1985multivariate,aravkin2016smart,menickelly2018robust} is to fit a model while simultaneously detecting and removing ``outlier" objectives $f_i$. Mathematically, we fix a number $h\in  \{1, \ldots, n\}$ indicating the number of ``inliers," and formulate the problem as follows: 
\begin{align*}
\minimize_{x \in \RR^d, w \in \RR^n} &~\sum_{i=1}^n w_if_i(x) \\
\text{subject to:} &~ w_i \in [0, 1] \text{ and }  \sum_{i=1}^n w_i = h.
\end{align*}
One can see that for fixed $x$, the only objective values that contribute to the sum are those that are among the $h$-minimal elements of the set $\{f_1(x), \ldots, f_n(x)\}$. In the Appendix, we provide a short proof that this objective is weakly convex. Notice that it is in general nonconvex, despite each $f_i$  begin convex. 
\end{example}

\section{Proximally Guided Stochastic Subgradient Method} \label{sec:Inexact}In this section, we formalize the proposed algorithm. First we slightly generalize the problem considered in the introduction, namely we assume that 
\begin{align}
\minimize_{x \in \RR^d} ~F(x) = \begin{cases}
f(x) & \text{if $x \in \cX$} \\
\infty & \text{otherwise.}
\end{cases}
\end{align}
where $f$ is a closed $\rho$-weakly convex function. Weak convexity of $f$ implies that each of the proximal subproblems $\min_{x \in \RR^d}\left\{ F(x) + (1/2\gamma)\|x- x_t\|^2\right\}$  is 
\begin{align*}
\mu:=\gamma^{-1} - \rho
\end{align*}
strongly convex. Next we introduce a stochastic subgradient oracle and a basic assumption on $F$. 
\begin{assumption}\label{assumption:main}
Fix a probability space $(\Omega, \cF, \PP)$ and equip $\RR^d$ with the Borel $\sigma$-algebra. Then we assume that
\begin{description}
	\item[(A1)] It is possible to generate IID\ realizations $z_1,z_2, \ldots$ from $P$.
	\item[(A2)] There is an open set $U\subseteq \RR^d$ containing $\cX$ and a measurable mapping $G : U \times \Omega \rightarrow \RR^d$ such that  $\EE_{z}[G(x,z)]\in \partial f(x)$.
	\item[(A3)] There is a constant $L \geq 0$ such that for all $x \in U$, we have $\EE_z\left[ \|G(x, z)\|^2\right] \leq L^2.$ 
	%\item[(A4)] The function $\phi$ is $\rho$-weakly convex on $X$ and $\phi_X^\ast := \inf_X \phi(x) > -\infty$.
\end{description}
\end{assumption}

Assumption~\ref{assumption:main} is standard in the literature on stochastic subgradient methods. In particular, assumptions (A1) and (A2) are identical to assumptions (A1) and (A2) in~\cite{doi:10.1137/070704277}, while assumption (A3) is identical to~\cite[Equation (2.5)]{doi:10.1137/070704277}. A useful consequence of (A3) is that $f$ itself is Lipschitz. 
\begin{lemma}[Lipschitz Continuity of $f${\cite[Section 3.2]{davis2018stochastic}}]\label{lem:Lipschitz}
Suppose that assumption (A3) holds. Then $f$ is $L$-Lipschitz continuous on U.
\end{lemma}

The main workhorse of PGSG is a stochastic subgradient method for solving  regularized subproblems $\min_{x \in \RR^d}\left\{ F(x) + (1/2\gamma)\|x- x_t\|^2\right\}$ induced by the proximal point method. We now state this method. 

\begin{algorithm}[H]
	\caption{Projected Stochastic Subgradient Method for Proximal Point Subproblems $\mathrm{PSSM}(y_0, G, \gamma, \{\alpha_t\}, J)$ \label{alg:SSM}}
	\begin{algorithmic}[1]
		\Require{$y_0 \in \cX,$ quadratic multiplier $\gamma> 0$, maximum iterations $J \in \NN$, nonnegative stepsize sequence $\{\alpha_t\}$.}
		\For{$j  = 0, \ldots, J-2$}
		\State Sample $z_{j}$ and set $v_{j} = G(y_{j},z_{j}) + \frac{1}{\gamma}(y_{j}-y_0)$%\in \partial (\obj^{\xi_{t,j}}(\cdot) + \frac{1}{2\gamma}\|\cdot - x_t\|^2)(y_{t,j})$
		%\Let{$\alpha_j$}{$\frac{2}{(\gamma^{-1}-\rho)(j+1) + 12/(\gamma-\gamma^2\rho)}$}
		\Let{$y_{j+1}$}{$\proj_\cX(y_{t,j} - \alpha_j v_{j})$}
		\EndFor
		\Ensure{ $\tilde y = \frac{2}{J(J+1)}\sum_{j=0}^{J-1}(j+1)y_{j}$.}
	\end{algorithmic}
\end{algorithm}

 Before introducing the Proximally Guided stochastic Subgradient (PGSG) method, we introduce two necessary algorithm parameters: 
\begin{align*}	
	j_t &\geq \frac{11}{\gamma^2\mu^2};\numberthis\label{eq:j_t_equation}\\
	\alpha_j &= \frac{2}{\mu\left(j+2 + \dfrac{36}{\gamma^4\mu^4(j+1)}\right)}. \numberthis\label{eq:alpha}
\end{align*}
The algorithm now follows.
\begin{algorithm}[H]
	\caption{Proximally Guided Stochastic Subgradient Method $ \mathrm{PGSG}(y_0, G, \gamma, \{\alpha_t\}, \{j_t\}, T)$ \label{alg:PGSG}}
	\begin{algorithmic}[1]
		\Require{$x_0 \in \cX$, weak convexity constant $\rho > 0$, $\gamma \in (0, 1/\rho)$, maximum iterations $T \in \NN$, stepsize sequence $\{\alpha_t\}$  \text{as in~\eqref{eq:alpha}},  maximum inner loop iteration $\{j_t\}$ \text{satisfying~\eqref{eq:j_t_equation}}.}
		\For{$t  = 0, \ldots, T-2$}
		\State{$x_{t+1} = \mathrm{PSSM}(x_t, G, \gamma, \{\alpha_t\}, j_t)$}
		\EndFor
		\Ensure{$x_R$, where $R$ is sampled uniformly from $\{0, \ldots, T-1\}$.}
	\end{algorithmic}
\end{algorithm}

As stated in the introduction PGSG employs an inner-outer loop strategy, which is shown in Algorithm~\ref{alg:PGSG}. The outer loop executes $T-1$ approximate proximal point steps, resulting in the iterates $\{x_t\}$. The inner loop, shown in Algorithm~\ref{alg:SSM}, approximately solves the proximal point subproblem, which is now strongly convex, using a stochastic subgradient method for strongly convex optimization~\cite{2012arXiv1212.2002L}. Beyond its use in governing the outer loop dynamics of PGSG, the proximal point subproblems also lead to a natural measure of stationarity. 

Indeed, for all $t \in \NN$, define the proximal point
\begin{align}\label{eq:definition_of_z}
\hat x_{t} := \argmin_{x\in \RR^d}\left\{F(x) + \frac{1}{2\gamma}\|x-x_t\|^2\right\}.
\end{align}
Note that $\hat x_t$ exists and is unique by the $\mu$-strong convexity of the proximal subproblem. We stress that this point, although in principle obtainable via convex optimization, is never computed. Instead it is only used to formulate convergence guarantees. To that end, the following Lemma shows that the gap $\gamma^{-1}\|x_t - \hat x_t\|$ is a natural measure of stationarity. %The proof of the Lemma is a simple exercise, so we omit it. 
\begin{lemma}[Convergence Criteria]\label{eq:dist_to_critical}
Let $F : \RR^d \rightarrow \RR \cup \{ \infty\}$ be a proper closed function. Let $x \in \RR^d$. If  
$$
\hat x \in \argmin_{y \in \RR^d} \left\{ F(y) + \frac{1}{2\gamma} \|y - x\|^2\right\},
$$
then we have the bound 
\begin{align}\label{eq:sub_dist}
\dist(x, \{ y\in \RR^d \mid \dist(0, \partial F(y))^2 \leq \gamma^{-2}\|x - \hat x\|^2\}) \leq \|x - \hat x\|^2.
\end{align}
\end{lemma}
\begin{proof}
As $\hat x$ is a minimizer, we have $$0 \in \partial \left[F(\cdot) + \frac{1}{2\gamma} \|\cdot - x\|^2\right](y) = \partial F(y) + \frac{1}{\gamma}(y - x),$$ 
where the second equality follows by the sum rule for a smooth additive term $(2\gamma)^{-1} \|\cdot - x\|^2$~\cite{rtrw}. Thus, we have the inclusion $\hat x \in \{ y\in \RR^d \mid \dist(0, \partial F(y))^2 \leq \gamma^{-2}\|x - \hat x\|^2\}$, which leads to the desired conclusion.
\end{proof}

Based on this Lemma, the iterate $x_t$ is $\varepsilon$-close to an $\varepsilon$-stationary point in expectation whenever 
$$
\EE\|x_t - \hat x_t\|^2 \leq \min \{ \varepsilon, \gamma^2 \varepsilon\}.
$$
Establishing this fact is the main technical goal of the following theorem.

%~\\
%Our main convergence theorem follows. 

\begin{theorem}[Convergence of PGSG]\label{thm:Convergence-With-Log} Let $x_0 \in \cX$, consider any $T \in \NN$, and let \break$x_R = \mathrm{PGSG}(x_0, G, \gamma, \{\alpha_t\}, \{j_t\}, T)$ Define the quantity
\begin{align*}
\cB_{T, \{j_t\}} := \frac{4}{T\mu}\left(F(x_0) - \inf F + \sum_{t=0}^{T-1} \frac{72L^2}{\mu(j_t+1)} \right).
\end{align*}
Then $\EE\|x_R- \hat x_R\|^2 \leq \cB_{T, \{j_t\}}$. Consequently,  we have the following bound: 
\begin{align*}
\EE\left[\dist(x_R, \{ x \mid \dist(0, \partial F(x))^2 \leq \gamma^{-2}\cB_{T, \{j_t\}}\} )^2\right] \leq \cB_{T, \{j_t\}}
\end{align*}
	In particular, given $\Delta \geq F(x_0) - \inf F$, and setting 
	\begin{align*}
	j_t :=  \left\lceil\max\left(\frac{576 L^2}{\mu^2\min\{\varepsilon, \varepsilon \gamma^2\}}, \frac{11}{\gamma^2\mu^2} \right)\right\rceil && \text{and} && T := \left\lceil\frac{4\Delta}{\mu\min\{\varepsilon, \varepsilon\gamma^2\}}\right\rceil,
	\end{align*}
	we have 
	\begin{align*}
	\EE\left[\dist(x_R, \{ x \mid \dist(0, \partial F(x))^2 \leq \varepsilon )^2\right] \leq \varepsilon.
	\end{align*}
	The total number of stochastic oracle evaluations required to compute this point is bounded by $j_t\cdot T = O(\Delta L^2\varepsilon^{-2})$. 
\end{theorem}

\begin{remark}[Obtaining a Nearly Stationary Point]\label{remark:finding_stat}As stated, the theorem indicates that $x_R$ is nearby a nearly stationary point. The proof of Lemma~\ref{eq:definition_of_z} shows that one can in principle obtain the nearly stationary  point $\hat x_R$ by solving the \emph{strongly convex} stochastic optimization problem 
$$
\hat x_R = \argmin_{x\in \RR^d}\left\{F(x) + \tfrac{1}{2\gamma}\|x-x_R\|^2\right\}, 
$$
which is solvable to any desired degree of accuracy (in expectation). Furthermore, Lemma~\ref{eq:definition_of_z} shows that one can estimate the degree of stationarity of $\hat x_R$ by the bound $\dist(0, \partial F(\hat x_R))^2 \leq \gamma^{-2}\|x_R - \hat x_R\|^2$. In particular, given an estimate, $\tilde x_{R} \approx \hat x_R$, we have the bound $\dist(0, \partial F(\hat x_R))^2 \leq 2\gamma^{-2}\|x_R - \tilde x_R\|^2 + 2\gamma^{-2} \|\tilde x_R - \hat x_R\|^2$, which indicates that $2\gamma^{-2}\|x_R - \tilde x_R\|^2$ may serve as a bound on the true stationarity of $\hat x_R$ (up to tolerance $2\gamma^{-2} \|\tilde x_R - \hat x_R\|^2$).
\end{remark}

\subsection{Proof of Theorem~\ref{thm:Convergence-With-Log}}
Throughout the proof we will need the following bound on the proximal point step: 
\begin{lemma}[Bounded Steplengths]\label{lem:prox_point_bound}
Let $\gamma > 0$, $x \in \cX$, and suppose that $$\hat x\in  \argmin_{y\in \RR^d}\left\{F(y) + \tfrac{1}{2\gamma}\|y-x\|^2\right\}.$$ Then $\gamma^{-1} \|x - \hat x\| \leq 2L$. 
\end{lemma}
\begin{proof}
Note that 
\begin{align*}
\frac{1}{2\gamma}\|x - \hat x \|^2 \leq F(x) - F(\hat x) \leq L\|x - \hat x \|,
\end{align*}
where Lipschitz continuity follows from Lemma~\ref{lem:Lipschitz}. Divide both sides of the inequality by $\tfrac{1}{2}\|x - \hat x\|$ to get the result. 
\end{proof}

We now analyze one inner loop of Algorithm~\ref{alg:PGSG}. This inner loop may be interpreted as a variant of the stochastic projected subgradient method applied to the strongly convex optimization problem, $$\minimize_{x \in \RR^d} \; F_y(x):= F(x) + \tfrac{1}{2\gamma}\|x - y\|^2,$$ 
We note that the following proof is similar in outline to~\cite{2012arXiv1212.2002L}, but the results of that work are not sufficient for our purposes.

%For convenience, we denote this regularized function in the $t$th outer iteration by $f_t(x):=\phi(x) + \frac{1}{2\gamma}\|x - x_t\|^2$. 
%Then the inner loop has the following guarantee, where $z_{t+1}$ is the unique minimizer of $f_t(x)$ over $X$.

\begin{proposition}[Analysis of PSSM]\label{prop:Inner-Convergence-With-Log}
	Let $y \in \cX$ and let $\hat y$ be the unique minimize of $F_y(x)$ over all $x \in \RR^d$. Set $\tilde y = \mathrm{PSSM}(y, G, \gamma, \{\alpha_t\}, J)$. Then if $\gamma \in (0, 1/\rho)$ and $\{\alpha_t\}$ is chosen as in~\eqref{eq:alpha}, we have%the unique minimizer of $f_t(x)$ over $X$
	\begin{align*}
	 \EE\left[F_y\left(\tilde y\right) - F_y(\hat y))\right]  &\leq  \frac{72L^2}{\mu(J+1)} +\frac{30\|y -\hat y\|^2}{\gamma^4\mu^3J(J+1)}; \\
	 \EE\left[\left\|\tilde y  - \hat y\right\|^2\right]  &\leq \frac{144L^2}{\mu^2(J+1)} +\frac{60\|y - \hat y \|^2}{\gamma^4\mu^4J(J+1)}; \\
	 \EE\left[\left\| y  - \tilde y\right\|^2\right]  &\leq  \frac{288L^2}{\mu^2(J+1)} +\left(2 + \frac{120}{\gamma^4\mu^4J(J+1)}\right)\|y - \hat y \|^2.
	\end{align*}
	On the other hand, if $0 < \alpha_{j}\leq 2\gamma $ for all $j$, but $\{\alpha_j\}$ and $\gamma$ are otherwise unconstrained, we have  $$\EE\left[\|y - \tilde y\|\right] \leq  L \sum_{i=0}^{J-1} \alpha_i.$$
\end{proposition}
 \begin{proof}
 Since $\hat y \in \cX$ and $\proj_\cX$ is nonexpansive, we have 
	\begin{align*}
	\|y_{j+1} - \hat y\|^2 & \leq \|y_{j} - \alpha_jv_{j} - \hat y\|^2\\
	& = \|y_{j} - \hat y\|^2 - 2\alpha_j\langle y_{j} -\hat y , v_{j}\rangle + \alpha_j^2\|v_{j}\|^2. \numberthis\label{eq:init_descent}
	\end{align*}
	To proceed further, we must bound $\|v_{t, j}\|^2$. To that end, recall that $F_y$ is $\mu$-strongly convex. Therefore, for any $x\in \cX$,
	\begin{align*}
		\EE_{z} \left\|G(x, z) + \frac{1}{\gamma}(x - y)\right\|^2&= \EE_{z} \left\|G(x, z) -\frac{1}{\gamma}(\hat y - y) +\frac{1}{\gamma}(x - \hat y)\right\|^2\\
		&\leq \EE_{z} 3\|G(x, z)\|^2 + 3\left\|\frac{1}{\gamma}(\hat y - y)\right\|^2 +3\left\|\frac{1}{\gamma}(x - \hat y)\right\|^2\\
		&\leq 15L^2 +3\left\|\frac{1}{\gamma}(x - \hat y)\right\|^2\\
		&\leq 15L^2 + \frac{6}{\gamma^2\mu}(F_y(x) - F_y(\hat y)),
	\end{align*}
	where the first inequality follows from Jensen's inequality, the second inequality uses (A3) twice and Lemma~\ref{lem:prox_point_bound}, and the third inequality follows from the strong convexity.
	
	Returning to Equation~\eqref{eq:init_descent}, we let $\bar v_{j} = \EE_{j} v_{j} \in \partial F_y(y_{j})$, where $\EE_j \left[\cdot \right]$ denotes the expectation conditioned on $y_1, \ldots, y_j$. Now, we take the conditional expectation of both sides of the equation, yielding 
	\begin{align*}
	\EE_{j}\|y_{j+1} - \hat y\|^2 & \leq \EE_{j}\|y_{j} - \hat y\|^2 - 2\alpha_j\langle y_{j} -\hat y , \bar v_{j}\rangle + \alpha_j^2\EE_{j}\|v_{j}\|^2\\
	& \leq \EE_{j}\|y_{j} - \hat y\|^2 + \alpha_j^2\left(15L^2 + \frac{6}{\gamma^2\mu}\EE_{j} F_y(y_{j}) -F_y(\hat y)\right) \\ & \indent - 2\alpha_j\left(\EE_{j} F_y(y_{j}) - F_y(\hat y) + \frac{\mu}{2}\EE_{j}\|y_{j} - \hat y\|^2\right)\\
	& = \left(1 - \alpha_j\mu\right)\EE_{j}\|y_{j} - \hat y\|^2 + 15\alpha_j^2L^2  - \left(2\alpha_j - \frac{6\alpha_j^2}{\gamma^2\mu}\right)\left(\EE_{j} F_y(y_{j}) - F_y(\hat y)\right)\\
	& \leq \left(1 - \alpha_j\mu\right)\EE_{t,j}\|y_{j} - \hat y\|^2 + 15\alpha_j^2L^2 - \alpha_j\left(\EE_{j} F_y(y_{j}) - F_y(\hat y)\right),
	\end{align*}
	where the second inequality uses our bound on $\EE_z\|G(x,z) + \gamma^{-1}(x-y)\|^2$ and the strong convexity of $F_y$, and the third inequality is a consequence of the bound: 
	$$
	\frac{6\alpha_j}{\gamma^2\mu} = \frac{2\mu(j+2)(6/\gamma^2\mu)}{(\mu(j+2))^2 + \frac{36}{\gamma^4\mu^2}\frac{j+2}{j+1}} \leq \frac{2\mu(j+2)(6/\gamma^2\mu)}{(\mu(j+2))^2 + (6/\gamma^2\mu)^2} \leq 1.
	$$
	Multiplying by $(j+1)/\alpha_j$, we find that 
	\begin{align*}
	(j+1)\alpha_j^{-1}\EE_{j}\|y_{j+1} - \hat y\|^2 & \leq (j+1)\left(\alpha_{j}^{-1} - \mu\right)\EE_{j}\|y_{j} - \hat y\|^2 + 15(j+1)\alpha_jL^2 \\ & \indent - (j+1)\left(\EE_{j} F_y(y_{j}) - F_y(\hat y)\right).
	\end{align*}
	By our choice of $\alpha_j$, we have $(j+1)\alpha_{j}^{-1} = (j+2)(\alpha_{j+1}^{-1}-\mu)$. Therefore, summing the previous inequality, we have
	\begin{align*}
	0 & \leq \left(\alpha_{0}^{-1} - \mu\right)\|y - \hat y\|^2 + 15L^2\sum_{j=0}^{J-1}(j+1)\alpha_j  - \sum_{j=0}^{J-1}(j+1)\left(\EE_{j} F_y(y_{j}) - F_y(\hat y)\right).
	\end{align*}
	Therefore, noting that $\sum_{j=0}^{j_t-1} (j+1)\alpha_j \leq 2j_t/\mu$ and $\alpha_0^{-1}-\mu = 18/(\gamma^4\mu^3)$, and using the convexity of $F_y$, we deduce
	\begin{align*}
	\EE\left(F_y(\tilde y) - F_y(\hat y)\right) & \leq \frac{36\|y - \hat y\|^2}{\gamma^4\mu^3J(J+1)} + \frac{60L^2}{\mu(J+1)}.
	\end{align*}
	The first distance bound then follows as a direct consequence of the strong convexity of $F_y$, while the second follows from the convexity of $\|\cdot\|^2$.
	
	Finally, we now work in the case in which $\gamma$ may be strictly greater than $1/\rho$. We claim that for all $j= 0, \ldots, J-1$, we have $\EE\left[ \|y_j - y_0\|\right] \leq L \sum_{i=0}^j \alpha_i$. Indeed, this is clearly true for $j = 0$. Inductively, we also have 
	\begin{align}
 		\EE_{j}\|y_{j+1} - x_t\| &\leq \EE_{j}\|y_{j} - \alpha_{j}\left(G(y_{j},z_{j}) + (y_{j}-x_t)/\gamma\right) - x_t\|\nonumber\\
 		&\leq |1 - \alpha_{j}/\gamma|\cdot \EE_{j}\|y_{j} - x_t\| + \alpha_{j}\EE_{\Xi_0}\|G(y_{j},z_{j})\|\nonumber\\
 		&\leq \EE_{j}\|y_{j} - x_t\| + \alpha_{j}L,\nonumber
 	\end{align}
	where the first inequality follows by nonexpansiveness of $\proj_{\cX}$ and the third follow from the inequality $0 < \alpha_{j}\leq 2\gamma$. Applying the law of expectation completes the inductive step. Therefore, we have 
	\begin{align*}
	\EE\left[\|y - \tilde y\|\right] \leq \EE\left[ \frac{2}{J(J+1)}\sum_{j=0}^{J-1}(j+1)\|y_j - y\|\right] &\leq  \frac{2}{J(J+1)}\sum_{j=0}^{J-1}(j+1)\left(L \sum_{i=0}^j \alpha_i\right) \\
	&\leq L \sum_{i=0}^{J-1} \alpha_i,
	\end{align*}
	as desired.

\end{proof}

We now give the proof of Theorem \ref{thm:Convergence-With-Log}.

 \begin{proof}[Proof of Theorem \ref{thm:Convergence-With-Log}]
  By the strong convexity of the proximal point subproblem, we have 
	\begin{align*}
	F(\hat x_t) \leq F(x_t) - \left(\frac{1}{2\gamma} + \frac{\mu}{2}\right)\|\hat x_t - x_t\|^2.
	\end{align*}
 	Then by Proposition~\ref{prop:Inner-Convergence-With-Log}, we have the following bound: 
	\begin{align*}
	\EE_{t}\left[F\left(x_{t+1}\right) \right]  &\leq F(\hat x_t) + \frac{1}{2\gamma} \|\hat x_t - x_t\|^2 +  \frac{72L^2}{\mu(j_t+1)} +\frac{30\|x_t -\hat x_t\|^2}{\gamma^4\mu^3j_t(j_t+1)} \\
	&\leq F(x_t) + \frac{72L^2}{\mu(j_t+1)} - \left(\frac{\mu}{2} - \frac{30}{\gamma^4\mu^3j_t(j_t+1)}\right)\|x_t - \hat x_t\|^2,
	\end{align*}
	where $\EE_t\left[\cdot\right]$ denotes the expectation conditioned on $x_1, \ldots, x_t$. Rearranging, using the lower bound on $j_t$ (which makes the multiple of $\|\hat x_t - x_t\|^2$ larger than $\mu/4$ as $30/121 < 1/4$), applying the law of total expectation, and summing, we find that 
	\begin{align*}
	\frac{1}{T}\sum_{t=0}^{T-1} \EE\left[\|x_t - \hat x_t\|^2 \right] &\leq \frac{4}{T\mu}\left(F(x_0) - \inf F + \sum_{t=0}^{T-1} \frac{72L^2}{\mu(j_t+1)} \right) ,
	\end{align*}
 as desired. To complete the proof, apply Lemma~\ref{eq:dist_to_critical}.
\end{proof}

\subsection{Probabilistic Guarantees}\label{sec:proabilistic}
In the previous section, we developed expected complexity results, which describe the average behavior of the PGSG over multiple runs. We are also interested in the behavior of a single run of the PGSG algorithm. Thus, in this section we recall the notion of an $(\varepsilon, \Lambda)$-solution given in the introduction: a random variable $\bar x$ is called an \emph{$(\varepsilon, \Lambda)$-solution} if 
$$
\PP\left(\dist(\bar x, \{ x \mid \dist(0, \partial F(x))^2 \leq \varepsilon\})^2 \leq \varepsilon  \right) \geq 1- \Lambda.
$$
Theorem~\ref{thm:Convergence-With-Log} together with Markov's inequality implies that $x_{R}$, generated with 
\begin{align*}
	j_t :=  \left\lceil\max\left(\frac{576 L^2}{\mu^2\min\{\varepsilon\Lambda, \varepsilon\Lambda \gamma^2\}}, \frac{12}{\gamma^2\mu^2} \right)\right\rceil && \text{and} && T := \left\lceil\frac{4\Delta}{\mu\min\{\varepsilon\Lambda, \varepsilon\Lambda\gamma^2\}}\right\rceil,
	\end{align*}
where $\Delta \geq F(x_0) - \inf F$, is an $(\varepsilon, \Lambda)$-solution after 
\begin{align}\label{eq:epsLambda1}
j_t \cdot T = O(\Delta L^2 (\varepsilon \Lambda)^{-2})
\end{align}
 stochastic oracle evaluations. In this section, we develop a two stage algorithm that significantly improves the dependence on $\Lambda$ in this bound. 

The method we propose proceeds in two phases. In the first phase, multiple independent copies of PGSG are called, resulting in candidates $x_{R^1}, \ldots, x_{R^S}$. For each of the candidates, we then compute an approximate proximal point $\tilde x_{R^s} \approx \hat x_{R^s}$. In the second phase, we select one of the candidates $x_{R^{\bar s}}$ based on the size of $\gamma^{-1} \|x_{R^s} - \tilde x_{R^s}\|$, a proxy for the true proximal step length. We will see that such a point is $(\varepsilon, \Lambda)$-solution, and the total number of stochastic oracle evaluations has a much better dependence on $\Lambda$. 

Before we introduce the algorithm, let us define three parameters 
	\begin{align}\label{eq:two_phase_parameters}
	j_t :=  \left\lceil\max\left\{\frac{576 L^2}{\mu^2\min\{\varepsilon/24, \varepsilon \gamma^2/24\}}, \frac{11}{\gamma^2\mu^2} \right\}\right\rceil, && \text{} && T := \left\lceil\frac{4\Delta}{\mu\min\{\varepsilon/24, \varepsilon\gamma^2/24\}}\right\rceil,
	\end{align}
	and 
	\begin{align*}
	J := \left\lceil\max \left\{ \frac{48L^2\sqrt{2}}{\mu\min\{\varepsilon, \varepsilon \gamma^2\}}\cdot \frac{S}{\Lambda},\frac{11}{\gamma^2 \mu^2} \cdot \sqrt{\frac{S}{\Lambda}}\right\} \right\rceil,
	\end{align*}
	where $\Delta \geq F(x_0) - \inf F$. The algorithm now follows. 

\begin{algorithm}[H]
	\caption{Two Phase Proximally Guided Stochastic Subgradient Method \label{alg:2-indices} $\mathrm{2PGSG}(x_0, G, \gamma, \{\alpha_t\}, \{j_t\}, T, J, S)$}
	\begin{algorithmic}[1]
		\Require{$x_0 \in \cX$, weak convexity constant $\rho > 0$, $\gamma \in (0, 1/\rho)$, maximum iterations $T \in \NN$ satisfying~\eqref{eq:two_phase_parameters}, stepsize sequence $\{\alpha_t\}$  \text{as in~\eqref{eq:alpha}},  maximum inner loop iteration $\{j_t\}$ \text{satisfying~\eqref{eq:two_phase_parameters}}, Stochastic Subgradient Iteration $J \in \NN$, number of copies $S \in \NN$.}
		\State \textbf{Optimization Phase}
		%\State \quad Let $R^1, \ldots, R^S$ be independent uniform random variables on $\{0, \ldots, T-1\}.$
		\For{$s  = 1, \ldots, S$}
		\State Set $x_{R^s}= \mathrm{PGSG}(x_0, G, \gamma, \{\alpha_t\}, \{j_t\}, T)$.
		\State Set $\tilde x_{R^s} = \mathrm{PSSM}(x_{R^s}, G, \gamma, \{\alpha_t\}, J)$.
		\EndFor
		\State \textbf{Post-Optimization Phase}
		\State Choose $ x^\ast = x_{R^{\bar s}}$ from the candidate list $\{ x^s_r\}_{s = 1}^S$ such that  
		\begin{align*}
		\bar s = \argmin_{s = 1, \ldots, S} ~\|x_{R^s} - \tilde x_{R^s}\| && && 
		\end{align*}
			\Ensure{ $x^\ast$}

	\end{algorithmic}
\end{algorithm}

The analysis of this algorithm requires a bound on the expectation of $\|x_{R^{s}} - \tilde x_{R^{s}}\|^2 $ and $\|\tilde x_{R^{s}} - \hat x_{R^{s}}\|^2$, which we now provide.
\begin{lemma}\label{lem:2Phase_bounds}
Let $x_{R^s}$ be generated as in Algorithm~\ref{alg:2-indices}. Then 
\begin{align*}
\EE\left[ \|x_{R^{s}} - \tilde x_{R^{s}}\|^2 \right] &\leq \frac{1}{4}\min\{\varepsilon,\gamma^2 \varepsilon\}; \\
\EE\left[ \|\tilde x_{R^{s}} - \hat x_{R^{s}}\|^2 \right] &\leq \frac{\Lambda}{4S}\min\{\varepsilon,\gamma^2 \varepsilon\}
\end{align*}
\end{lemma}
\begin{proof}
By Proposition~\ref{prop:Inner-Convergence-With-Log} and Theorem~\ref{thm:Convergence-With-Log}, the bound holds: 
\begin{align*}
\EE\left[ \|x_{R^{s}} - \tilde x_{R^{s}}\|^2 \right] &\leq \frac{288L^2}{\mu^2(J+1)} +\left(2 + \frac{120}{\gamma^4\mu^4J(J+1)}\right)\EE\left[\|x_{R^s}  - \hat x_{R^s} \|^2 \right]\\
&\leq \frac{288L^2}{\mu^2(J+1)} +\left(2 + \frac{120}{\gamma^4\mu^4J(J+1)}\right)\cB_{T, \{j_t\}} \\
%&\leq \frac{288L^2}{\mu^2(J+1)} +\left(2 + \frac{120}{\gamma^4\mu^4J(J+1)}\right)\cB_{T, \{j_t\}} \\
&\leq \frac{\Lambda}{8S}\min\{\varepsilon,\gamma^2 \varepsilon\}/8 + \min\{\varepsilon,\gamma^2 \varepsilon\}/8 \leq \min\{\varepsilon,\gamma^2 \varepsilon\}/4, 
\end{align*}
which proves the first bound. 

On the other hand, Proposition~\ref{prop:Inner-Convergence-With-Log} and Theorem~\ref{thm:Convergence-With-Log} imply that
\begin{align*}
\EE\left[ \|\tilde x_{R^{s}} - \hat x_{R^{s}}\|^2 \right]  &\leq \frac{144L^2}{\mu^2(J+1)} +\frac{60}{\gamma^4\mu^4J(J+1)}\EE\left[\|x_{R^s} - \hat x_{R^s} \|^2\right]\\
&\leq \frac{144L^2}{\mu^2(J+1)} +\frac{60}{\gamma^4\mu^4J(J+1)}\cB_{T, \{j_t\}} \\
&\leq \frac{\Lambda}{8S}\min\{\varepsilon,\gamma^2 \varepsilon\} + \frac{\Lambda}{8S}\min\{\varepsilon,\gamma^2 \varepsilon\} = \frac{\Lambda}{4S}\min\{\varepsilon,\gamma^2 \varepsilon\},
\end{align*}
which proves the second bound and completes the proof. 
\end{proof}

We now state the convergence guarantees for Algorithm.

\begin{theorem}\label{thm:2phase}
Let $x_0 \in \cX$ and let $S = \log_2(2/\Lambda)$. Then  \break$x^\ast = \mathrm{2PGSG}(x_0, G, \gamma, \{\alpha_t\}, \{j_t\}, T, J, S)$ returned by Algorithm~\ref{alg:2-indices} is an $(\varepsilon, \Lambda)$-solution. The total number of stochastic oracle evaluations called by Algorithm~\ref{alg:2-indices} is equal to 
\begin{align}\label{eq:epsLambda2}
S \cdot (j_t \cdot T + J) = O\left(\frac{\log_2(1/\Lambda) \Delta L^2}{\varepsilon^2} + \frac{\log_2(1/\Lambda) L^2}{\varepsilon \Lambda}\right).
\end{align}
\end{theorem}
\begin{proof}
By Lemma~\ref{eq:dist_to_critical}, it suffices to show that 
\begin{align*}
\PP\left( \| x^\ast - \hat  x^\ast\|^2 \leq \min\{\varepsilon, \gamma^2\varepsilon \} \right) \geq 1- \Lambda.
\end{align*}
To that end, note that 
\begin{align*}
\|x^\ast - \hat  x^\ast\|^2 &= \|(x_{R^{\bar s}} - \hat x_{R^{\bar s} })\|^2\\
&\leq 2\|x_{R^{\bar s}} - \tilde x_{R^{\bar s} }\|^2 + 2\|\tilde x_{R^{\bar s} } - \hat x_{R^{\bar s} }\|^2\\
&\leq 2\min_{s = 1, \ldots, S}\|x_{R^{s}} - \tilde x_{R^{s} }\|^2 + 2\max_{s = 1, \ldots, S}\|\tilde x_{R^{ s} } - \hat x_{R^{ s}}\|^2.
\end{align*}
Therefore, we have 
\begin{align*}
&\PP\left(\|x^\ast - \hat  x^\ast\|^2 \geq\min\{\varepsilon, \gamma^2\varepsilon \} \right)\\
&\leq \PP\left\{\min_{s = 1, \ldots, S}\|x_{R^{s}} - \tilde x_{R^{s}}\|^2 \geq \frac{1}{2}\min\{\varepsilon, \gamma^2\varepsilon \}\right) \\
&\hspace{20pt}+ \PP\left(\max_{s = 1, \ldots, S}\|\tilde x_{R^{ s} } - \hat x_{R^{ s} }\|^2 \geq \frac{1}{2}\min\{\varepsilon, \gamma^2\varepsilon \}\right).
\end{align*}

Notice that by Markov's inequality, independence, and Proposition~\ref{lem:2Phase_bounds}, we have:
\begin{align*}
\PP\left(2\min_{s = 1, \ldots, S}\|x_{R^{s}} - \tilde x_{R^{s} }\|^2 \geq \frac{1}{2}\min\{\varepsilon, \gamma^2\varepsilon \}\right) \leq 2^{-S} \leq \frac{\Lambda}{2}.
\end{align*}
On the other hand, by Markov's inequality, a union bound, and Proposition~\ref{lem:2Phase_bounds}, we have
\begin{align*}
&\PP\left(2\max_{s = 1, \ldots, S}\|\tilde x_{R^{ s}} - \hat x_{R^{ s}}\|^2 \geq\frac{1}{2}\min\{\varepsilon, \gamma^2\varepsilon \} \right) \leq \frac{\Lambda}{2}, %\\
%&=\PP\left(2\max_{s = 1, \ldots, S}\|\tilde x_{R^{ s}} - \hat x_{R^{ s}}\|^2 \geq\frac{2S}{\Lambda} \cdot \frac{\Lambda}{4S}\min\{\varepsilon, \gamma^2\varepsilon \} \right) , 
\end{align*}
which shows that $x^\ast$ is an $(\varepsilon, \Lambda)$-solution. 
\end{proof}

When the second term in~\eqref{eq:epsLambda2} is dominating, the obtained bound~\eqref{eq:epsLambda2} is \break $\log_2(2/\Lambda)/\varepsilon\Lambda$ times smaller than the bound~\eqref{eq:epsLambda1} obtained by the PGSG algorithm.

\subsection{PGSG with Unknown Weak Convexity Constant}\label{sec:weakconvexityunknown} Algorithm~\ref{alg:PGSG} requires that the parameters $\varepsilon$, $L$, and $\rho$ are known. In practice, computing $L$ and $\rho$ may be nontrivial. In this section we show that a simple strategy---letting $j_t$ tend to infinity and $\gamma_t$ tend to zero---results in a sublinear convergence rate without knowledge of any problem parameters. We formalize this procedure in Algorithm~\ref{alg:PGSG-free} using the following parameters: fix a  hyper-parameter $0<\beta<1$, and define 
\begin{align*}
	\gamma_t &:= (t+1)^{-\beta}; \numberthis\label{eq:diminishing-gamma}\\
		j_t &:=  t + 44; \numberthis\label{eq:diminishing-j_t}\\
		\alpha_{t,j} &:= \frac{4\gamma_t}{j +1 + \frac{288}{j+1}}. \numberthis\label{eq:diminishing-alpha}	
\end{align*}
The algorithm now follows. 
\begin{algorithm}[H]
	\caption{Parameter Free Proximally Guided Stochastic Subgradient Method $ \mathrm{PFPGSG}(y_0, G, \{\gamma_t\}, \{\alpha_t\}, \{j_t\}, T)$ \label{alg:PGSG-free}}
	\begin{algorithmic}[1]
		\Require{$x_0 \in \cX$,  $\{\gamma_t\}$ satisfying~\eqref{eq:diminishing-gamma}, maximum iterations $T \in \NN$, stepsize sequence $\{\alpha_t\}$  \text{as in~\eqref{eq:diminishing-alpha}},  maximum inner loop iteration $\{j_t\}$ \text{satisfying~\eqref{eq:diminishing-j_t}}.}
		\For{$t  = 0, \ldots, T-2$}
		\State{$x_{t+1} = \mathrm{PSSM}(x_t, G, \gamma_t, \{\alpha_t\}, j_t)$}
		\EndFor
		\Ensure{$x_R$, where $R$ is sampled with probability $\PP(R = t) \propto \gamma_t$ from $\{0, \ldots, T-1\}$.}
	\end{algorithmic}
\end{algorithm}

In the following, we establish convergence guarantees for the parameter free variant of PGSG. The proof splits the analysis of PFPGSG into two parts. In the first part, $\gamma_t \geq 1/\rho$. In this setting, the analysis of the previous section does not apply. Thus, we show that that the iterates do not wander very far. In the second part, $\gamma_t \leq 1/\rho$, and an argument similar to the one presented in Theorem~\ref{thm:Convergence-With-Log} applies. Combining these results then leads to the theorem. To that end, we address the first part now. 
\begin{lemma}
Let $T_0 = \lceil (2\rho)^{1/\beta}\rceil$. Then 
$$
\EE\left[ F(x_{T_0})\right] \leq  F(x_0) + L^2 T_0 \log(T_0 + 125).
$$
\end{lemma}
\begin{proof}
By Proposition~\ref{prop:Inner-Convergence-With-Log}, as $\alpha_j < 2\gamma_t$, we have we have $\EE_{T_0}\|x_{t+1} - x_{t}\| \leq L\sum_{j= 0}^{j_t-1} \alpha_j$ for all $t = 0, \ldots, T_0-1$.   	\begin{align}
 		 \EE_{T_0}F(x_{T_0}) & \leq F(x_0) + L\EE_{T_0}\|x_{T_0} - x_0\|\nonumber\\
 		 & \leq F(x_0) + L\sum_{t=0}^{T_0-1}\EE_{T_0}\|x_{t+1} - x_t\|\nonumber\\
 		 %& \leq \phi(x_0) + L\sum_{t=0}^{T_0-1}\sum_{j=0}^{j_t-1} \EE_{\Xi_0}\|\alpha_{t,j}G(y_{j,t},\xi_{j,t}) + \frac{1}{\gamma_t}(y_{j,t} - x_t) -x_t\|\nonumber\\
 		 & \leq F(x_0) + L^2\sum_{t=0}^{T_0-1}\sum_{j=0}^{j_t-1} \alpha_{t,j}\nonumber\\
		 & \leq F(x_0) + L^2\sum_{t=0}^{T_0-1}\sum_{j=0}^{j_t-1} \frac{4}{j+1}\nonumber\\
 		 & \leq F(x_0) + L^2T_0\log(T_0+125) \label{eq:objective-bound},
 	\end{align}
as desired. 
\end{proof}

We now address the second part of the argument, and with it, deduce the following theorem. At first glance, the presented rate appears to be better than the rate obtained by Algorithm~\ref{alg:PGSG}, which requires knowledge of $\rho$. However, it is not because the factor $\gamma_R^{-2} = (R+1)^{2\beta}$ is no longer a constant. Instead, the convergence rate of Algorithm~\ref{alg:PGSG-free} is on the order of $O(T^{1-\beta})$ in the worst case.

\begin{theorem}[Convergence of Parameter Free PGSG]\label{thm:PF-PGSG} Let $T_0 = \lceil (2\rho)^{1/\beta}\rceil$. And consider any $T \in \NN$. Let $x_R = \mathrm{PFPGSG}(x_0, G, \{\gamma_t\}, \{\alpha_t\}, \{j_t\}, T)$. Define the quantity
\begin{align*}
\cC_{T, \{j_t\}} := \frac{8(1+\beta)}{(T+1)^{1+\beta}}\left(F(x_{0}) - \inf F+ (144C  +  T_0 \log(T_0 + 125) + \tfrac{T_0}{2})L^2\right),
\end{align*}
where $C:= \sum_{t=T_0}^{\infty} t^{-1-\beta} < \infty$. Then $\EE\|x_R- \hat x_R\|^2 \leq \cC_{T, \{j_t\}}$. Consequently,  we have the following bound: 
\begin{align*}
\EE\left[\dist(x_R, \{ x \mid \dist(0, \partial F(x))^2 \leq (R+1)^{2\beta}\cC_{T, \{j_t\}}\} )^2\right] \leq \cC_{T, \{j_t\}}.
\end{align*}
\end{theorem}
\begin{proof}
Suppose that $t \geq T_0$ and notice this ensures $\gamma_t \in (0, 1/\rho)$. Following an argument nearly identical to the proof of Theorem~\ref{thm:Convergence-With-Log}, we find that for all $t \geq T_0$, we have
\begin{align*}
\EE_{t}\left[F\left(x_{t+1}\right) \right] 
	&\leq F(x_t) + \frac{72L^2}{\mu_t(j_t+1)} - \left(\frac{\mu_t}{2} - \frac{30}{\gamma_t^4\mu_t^3j_t(j_t+1)}\right)\|x_t - \hat x_t\|^2,
\end{align*}
where $\mu_t = \gamma_t^{-1} - \rho$ and $\EE_t\left[\cdot\right]$ denotes the expectation conditioned on $x_1, \ldots, x_t$. We now show that the coefficient of $-\|x_t - \hat x_t\|^2$ is greater than or equal to $\mu_t/4$. Indeed, it suffices to show that $j_t \geq 12/(1-\gamma_t \rho)^2$. To that end, note that
$$
\gamma_t \leq (\lceil (2\rho)^{1/\beta}\rceil + 1)^{-\beta} \leq 1/(2\rho).
$$
Therefore, $1- \gamma_t \rho \geq 1/2$, which leads to the claimed inequality: $12/(1-\gamma_t \rho)^2 \leq 44 \leq j_t$. 

Using the lower bound $\mu_t \geq 1/(2\gamma_t)$ (which follows because $t \geq T_0$), we thus find 
\begin{align*}
\sum_{t=T_0}^{T-1}\frac{1}{8\gamma_t} \EE\left[\|x_t - \hat x_t\|^2 \right] &\leq\EE\left[F(x_{T_0}) - \inf F\right] + \sum_{t=T_0}^{T-1} \frac{144\gamma_tL^2}{(j_t+1)} \\
&\leq F(x_{0}) - \inf F+ \sum_{t=T_0}^{T-1} \frac{144\gamma_tL^2}{(j_t+1)}  + L^2 T_0 \log(T_0 + 125).
\end{align*}
We would like to extend the sum on the left hand side of the previous inequality to all $t$ between $0$ and $T-1$. To that end, we bound the excess terms 
\begin{align*}
\sum_{t=0}^{T_0 - 1}\frac{1}{8\gamma_t}\EE\left[\|x_t - \hat x_t\|^2 \right] \leq \sum_{t=0}^{T_0 - 1}\frac{\gamma_tL^2}{2} \leq \frac{T_0L^2}{2}.
\end{align*}

Therefore, using the bounds $\sum_{t=T_0}^{\infty}\gamma_t/(j_t+1) \leq \sum_{t=T_0}^{\infty} t^{-1-\beta} = C < \infty$ and $\sum_{t=0}^{T-1} \gamma_t^{-1} \geq \int_{-1}^{T-1}(t+1)^\beta dt = T^{1+\beta}/(1+\beta)$, we have, 
\begin{align*}
&\EE\left[\|x_R - \hat x_R\|^2 \right]\\
&=\frac{1}{\sum_{t=0}^{T-1} \gamma_t^{-1}}\sum_{t=0}^{T-1}\frac{1}{\gamma_t} \EE\left[\|x_t - \hat x_t\|^2 \right] \\
&\hspace{20pt}\leq \frac{8(1+\beta)}{(T+1)^{1+\beta}}\left(F(x_{0}) - \inf F+ 144CL^2  + L^2 T_0 \log(T_0 + 125) + \frac{T_0L^2}{2}\right),
\end{align*}
as desired. To complete the proof, apply Lemma~\ref{eq:dist_to_critical}.
\end{proof}

\section{Experimental Results} \label{sec:experiments}

In this section we address the population version of the robust real phase retrieval problem: fix a vector $\bar x \in \RR^d$ and define
\begin{align}\label{eq:rob_phase}
F(x) := \EE_{a, \delta, \xi}\left[ | \dotp{a, x}^2 - (\dotp{a, \bar x}^2 + \delta \cdot \xi)|\right],
\end{align}
where $a, \delta,$ and $\xi$ are independent random variables satisfying the following assumptions
\begin{enumerate}
\item[(B1)] $a$ is a zero mean standard Gaussian random variable in $\RR^d$; 
\item[(B2)]  $\delta$ is a $\{0, 1\}$-random variable with $P(\delta = 1) = 0.25$; 
\item[(B3)] $\xi$ is a zero mean Laplace random variable with scale parameter $1$.
\end{enumerate}
In this setting, it is possible to show that the only minimizers of $F(x)$ are $\pm \bar x$~\cite[Lemma B.8]{davis2017nonsmooth}. In Lemma~\ref{lem:weak_phase}, we show that this function is 2-weakly convex.

\textbf{Implementation.} Each step of PGSG and the stochastic subgradient method requires access to a subgradient of a random function of the form
$$
f(x, a, \delta, \xi) = | \dotp{a, x}^2 - (\dotp{a, \bar x}^2 + \delta \cdot \xi)|.
$$
We choose the selection operator 
$$
G(x, a, \delta, \xi) = 2\dotp{a, x} a\cdot  \sign(\dotp{a, x}^2 - (\dotp{a, \bar x}^2 + \delta \cdot \xi)) \in \partial_x f(x, a, \delta, \xi).
$$
It is a straightforward exercise to show that $G$ satisfies assumption~\ref{assumption:main} on any bounded set $\cX$. For our purposes we choose $\cX$ to be a closed ball with a large radius, $r = 10^6$. In our experiments, we never had to explicitly enforce this constraint. 

\textbf{Experiment 1: Sensitivity to Stepsize.} In the first experiment we compare the performance of PGSG to the stochastic subgradient method, which possessed no complexity guarantees at the time of writing this manuscript. In the stochastic subgradient method, we choose stepsizes of the form $\gamma/(t+10)^\beta$ for varying $\gamma>0$ and $\beta\in\{1/2,1\}$. For PGSG, we chose varying values of $\gamma>0$ and then set $\alpha_j$ by~\eqref{eq:alpha}, $j_t = 250$, and $\mu = 1/2\gamma$.
Figure~\ref{fig:plot} shows the result of running these two methods to solve robust real phase retrieval problems with $d=50$.

\begin{figure}[t]
	\centering
	\begin{tabular}{cc}
		\includegraphics[width=.46\textwidth]{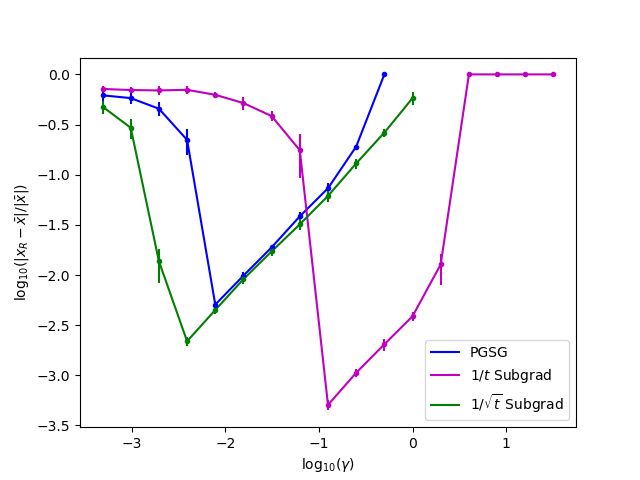} &
		\includegraphics[width=.46\textwidth]{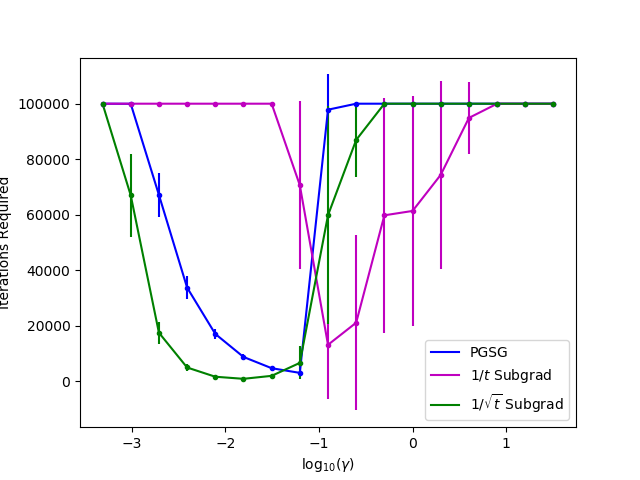} \\
		(a) & (b)\\		
	\end{tabular}
	\caption{Performance of PGSG and the subgradient method for values of $\gamma$ averaged over 50 trials. Error bars are included to show one standard deviation. Plot (a) shows the relative distance to a minimizer after 25000 subgradient evaluations. Plot (b) shows the number of subgradient evaluations needed until the relative distance $0.05$ to a minimizer.}
	\label{fig:plot}
\end{figure}

\textbf{Experiment 2: Mean and Variance of Solution Estimates.} 
Unlike the subgradient method, PGSG provides an easily computed estimate of the of how close $x_{R}$ is to a nearly stationary point; see the discussion surrounding Lemma~\ref{eq:sub_dist} and Remark~\ref{remark:finding_stat}. For PGSG, 2PGSG, and PFPGSG, this is given by $\gamma^{-1}\|x_{R}-x_{R+1}\|$, $\gamma^{-1}\|x_{R^s}-\tilde{x}_{R^s}\|$, and $\gamma^{-1}_R\|x_{R}-x_{R+1}\|$ respectively. Proposition~\ref{prop:Inner-Convergence-With-Log} shows these estimates are close to $\gamma^{-1}\|x_R - \hat x_R\|$ in expectation, which, according to Lemma~\ref{eq:sub_dist}, is a natural measure of stationarity. Using these stationarity measures, we analyze the numerical performance of the three algorithms proposed in this manuscript.  

Based on the results of Experiment 1,  we set $\gamma=2^{-6}$ for the PGSG and 2-PGSG algorithms. We furthermore set  $\alpha_j$ by~\eqref{eq:alpha} and let $\mu = 1/2\gamma$. For both methods, we consider two different selections for the number of inner iterations $j_t \in\{10^3, 10^4\}$. These choices determine the level of stationarity reached by the algorithm. For 2PGSG, we fix $S=5$ and $J= 5T$. For PFPGSG, we set $\beta=1/2$, $\gamma_t = (t+1)^{-\beta}/10$ (which differs from~\eqref{eq:diminishing-gamma} by a factor of ten), $j_t$ as in~\eqref{eq:diminishing-j_t}, and $\alpha_j$ as in~\eqref{eq:diminishing-alpha}.

Table~\ref{tab:stationarity} lists the mean and variance of the stationarity measures averaged over $50$ trials. Each sub column shows the performance of the target algorithm as the computational budget increases. We find that with $j_t=10^3$, both PGSG and 2PGSG quickly converge to a region of stationarity and then do not improve.  With $j_t=10^4$, both of these methods reach a level of stationarity an order of magnitude smaller than with the choice $j_t=10^3$. Under sufficiently large computational budget ($2500000$ stochastic subgradient evaluations), the variance of the stationarity reported by 2PGSG is consistently lower than that of PGSG as expected from Theorem~\ref{thm:2phase}. Finally, we note that the performance of PFPGSG is similar to PGSG in most regimes.

%The mean and variance o
%
%f our estimate of stationarity for each method over $50$ trials is given in Table~\ref{tab:stationarity}. Considering the columns corresponding to each algorithm shows its performance as the computational budget increases. We find that with $j_t=10^3$, both PGSG and 2PGSG quickly converge to a region of stationarity and then do not improve. With $j_t=10^4$, both of these methods reach a level of stationarity an order of magnitude smaller than with $j_t=10^3$.
%Under sufficiently large computational budget ($2500000$ stochastic subgradient evaluations), the variance of the stationarity reported by 2PGSG is consistently lower than that of PGSG as expected from our theory. %The parameter-free variant seems to be competitive at every level of problem dimension and computational budget.

\begin{table} \label{tab:stationarity}
\begin{tabular}{|c|c|c|c|c|c|c|}
	\hline
	\multirow{2}{*}{Oracle Calls} & & PGSG & PGSG & 2PGSG & 2PGSG & \multirow{2}{*}{PFPGSG}\\
	& & $j_t=1000$ & $j_t=10000$ & $j_t=1000$ & $j_t=10000$ & \\\hline
	& & \multicolumn{5}{c}{$d=50$}\\\hline
	\multirow{2}{*}{100000} & mean & 1.538 & 10.02 & 1.099 & 12.46 & 2.877\\
	& var.  & 0.0380 & 1.683 & 0.0153 & 5.871 & 0.178\\\hline
	\multirow{2}{*}{500000} & mean & 1.492 & 0.2043 & 1.024 & 8.406 & 1.615\\
	& var.  & 0.0542 & 9.27e-4 & 0.0119 & 0.669 & 0.0421\\\hline
	\multirow{2}{*}{2500000} & mean & 1.575 & 0.2083 & 1.034 & 0.1331 & 0.847\\
	& var.  & 0.0600 & 7.53e-4 & 0.0152 & 2.562e-4 & 0.0128\\\hline
	
	& & \multicolumn{5}{c}{$d=100$}\\\hline
	\multirow{2}{*}{100000} & mean & 3.632 & 17.12 & 2.703 & 23.04 & 6.625\\
	& var.  & 0.137 & 3.287 & 0.0544 & 6.117 & 0.361\\\hline
	\multirow{2}{*}{500000} & mean & 3.579 & 3.678 & 2.534 & 11.83 & 3.815\\
	& var.  & 0.145 & 22.35 & 0.0430 & 0.891 & 0.145\\\hline
	\multirow{2}{*}{2500000} & mean & 3.622 & 0.540 & 2.564 & 0.365 & 2.121\\
	& var.  & 0.127 & 2.71e-3 & 0.0468 & 1.01e-3 & 0.0380\\\hline
	
	& & \multicolumn{5}{c}{$d=500$}\\\hline
	\multirow{2}{*}{100000} & mean & 27.67 & 76.86 & 24.32 & 100.7 & 41.65\\
	& var.  & 8.843 & 16.95 & 2.465 & 20.31 & 6.471\\\hline
	\multirow{2}{*}{500000} & mean & 25.53 & 23.52 & 17.13 & 42.25 & 25.59\\
	& var.  & 1.519 & 1.474 & 0.341 & 4.772 & 1.946\\\hline
	\multirow{2}{*}{2500000} & mean & 25.64 & 4.759 & 17.10 & 3.519 & 15.09\\
	& var.  & 1.236 & 0.0454 & 0.374 & 0.0118 & 0.452\\\hline
	
	& & \multicolumn{5}{c}{$d=1000$}\\\hline
	\multirow{2}{*}{100000} & mean & 64.73 & 156.5 & 34.36 & 199.5 & 59.25\\
	& var.  & 14.37 & 49.09 & 2.388 & 53.92 & 167.2 \\\hline
	\multirow{2}{*}{500000} & mean & 55.97 & 40.99 & 33.61 & 86.97 & 54.48 \\
	& var.  & 3.426 & 2.091 & 0.890 & 9.233 & 71.38 \\\hline
	\multirow{2}{*}{2500000} & mean & 55.27 & 11.88 & 33.40 & 9.008 & 33.86 \\
	& var.  & 4.854 & 0.119 & 0.634 & 0.055 & 1.350 \\\hline
\end{tabular}
\caption{Estimated stationarity level for each of the proposed algorithms averaged over $50$ trails.}
\end{table}

\appendix

\section{Trimmed Estimation}

\begin{proposition}
Suppose that $f_1, \ldots, f_n$ are convex, $L$-Lipschitz continuous functions on $\RR^d$. Then the objective 
\begin{align*}
F(w, x) = \begin{cases}
\frac{1}{n}\sum_{i=1}^n w_i f_i(x) & \text{if $w_i \in [0, 1]$ and $\sum_{i=1}^n w_i = h$;}\\
\infty & \text{otherwise.}
\end{cases}
\end{align*}
is $L$-weakly convex. 
\end{proposition}
\begin{proof}
We argue using Proposition~\ref{prop:generalweak}. Let $(w, x), (\tilde w, \tilde x) \in \dom ~F$ and let $\lambda \in [0, 1]$. Then 
\begin{align*}
&F( (1-\lambda)(w, x) + \lambda(\tilde w, \tilde x)) \\
&=\frac{1}{n} \sum_{i=1}^n  ((1-\lambda)w_i + \lambda\tilde w_i) f_i( (1-\lambda)x + \lambda \tilde x)\\
&= \frac{1}{n}\sum_{i=1}^n  (1-\lambda)w_i  f_i( (1-\lambda)x + \lambda \tilde x) +\frac{1}{n} \sum_{i=1}^n  \lambda\tilde w_i f_i( (1-\lambda)x + \lambda \tilde x)\\
&\leq \frac{1}{n}\sum_{i=1}^n  (1-\lambda)w_i  ((1-\lambda)f_i( x )+ \lambda f_i( \tilde x)) + \frac{1}{n}\sum_{i=1}^n   \lambda\tilde w_i ((1-\lambda)f_i( x )+ \lambda f_i( \tilde x))\\
&= \frac{1}{n}\sum_{i=1}^n  (1-\lambda)w_i f_i( x )+\frac{1}{n} \sum_{i=1}^n\lambda (1-\lambda)w_i  (f_i( \tilde x) - f_i(x)) \\
&\hspace{20pt}+ \frac{1}{n}\sum_{i=1}^n   \lambda\tilde w_if_i(\tilde w_i) + \frac{1}{n}\sum_{i=1}^n   \lambda(1-\lambda)\tilde w_i ((f_i( x )-  f_i( \tilde x))\\
&= (1-\lambda) F(w, x) + \lambda F(\tilde w, \tilde x)+ \frac{1}{n}\lambda(1-\lambda)\sum_{i=1}^n  ( \tilde w_i - w_i) ((f_i( x )-  f_i( \tilde x))\\
&\leq (1-\lambda) F(w, x) + \lambda F(\tilde w, \tilde x)  + \frac{\lambda(1-\lambda)L}{2}\|w - \tilde w\|^2 + \frac{\lambda(1-\lambda)L}{2}\|x- \tilde x\|^2,
\end{align*}
as desired. 
\end{proof}

\section{Weak Convexity of Robust Phase Retrieval}

\begin{lemma}\label{lem:weak_phase}
The robust phase retrieval loss defined in~\eqref{eq:rob_phase} is $2$-weakly convex.
\end{lemma}
\begin{proof}
For all $x, y, a \in \RR^d$, we have
$$
\dotp{a, \lambda x + (1-\lambda) y}^2 = \lambda \dotp{a, x}^2 + (1-\lambda)\dotp{a, y}^2 - \lambda(1-\lambda)\dotp{a, y - x}^2.
$$
Thus, we have
\begin{align*}
&F(\lambda x + (1-\lambda )y) \\
&= \EE_{a, \delta, \xi}\left[ | \dotp{a,  \lambda x + (1-\lambda)y}^2 - (\dotp{a, \bar x}^2 + \delta \cdot \xi)|\right]\\
&\leq \lambda F(x) + (1-\lambda) F(y) + \lambda(1-\lambda)\EE_a\left[ \dotp{a, y-x}^2 \right] \\
&= \lambda F(x) + (1-\lambda) F(y) + \lambda(1-\lambda)\|x - y\|^2.
\end{align*}
Therefore, by Proposition~\ref{prop:generalweak}, $F$ is $2$-weakly convex. 
\end{proof}

\section*{Acknowledgments}
We thank Dmitriy Drusvyatskiy, George Lan, and the two anonymous reviewers for helpful comments.

\bibliographystyle{siamplain}
\bibliography{references}
\end{document}